\documentclass[12pt]{article} 

\usepackage[margin=1in]{geometry}

\usepackage{setspace}
\usepackage{xcolor}
\usepackage{subcaption}

\usepackage{mathtools} 

\usepackage{amssymb,amsfonts,amsmath,latexsym,amsthm}
\usepackage[]{graphicx}
\usepackage{bm}
\usepackage{dsfont}
\usepackage{xspace}

\usepackage[round,comma]{natbib}

\usepackage[utf8]{inputenc}
\usepackage{amsmath}
\usepackage{amsthm}
\usepackage{amssymb}
\usepackage{amsfonts}
\usepackage[colorlinks,citecolor=blue,linkcolor=blue,urlcolor=blue]{hyperref}

\newtheorem{theorem}{Theorem}[section]
\newtheorem*{theorem*}{Theorem} 
\newtheorem{corollary}[theorem]{Corollary}
\newtheorem{lemma}[theorem]{Lemma}

\newtheorem{proposition}[theorem]{Proposition}

\theoremstyle{remark}
\newtheorem*{remark}{Remark}

\allowdisplaybreaks 

\DeclareMathOperator*{\Ber}{Ber}
\DeclareMathOperator*{\Beta}{Beta}
\DeclareMathOperator*{\Tri}{Tri}

\DeclareMathOperator*{\Bin}{Bin}
\DeclareMathOperator*{\Unif}{Unif}

\newcommand{\R}{\mathbb{R}}

\newcommand{\E}{\mathbb{E}}

\newcommand{\edr}{\mathrm{e}}

\newcommand\opt{\mathrm{opt}}

\DeclareMathOperator{\Var}{\mathbb{V}}

\newcommand{\beq}{\begin{equation}}
\newcommand{\eeq}{\end{equation}}
\newcommand{\beqq}{\begin{equation*}}
\newcommand{\eeqq}{\end{equation*}}
\newcommand{\bea}{\begin{eqnarray}}
\newcommand{\eea}{\end{eqnarray}}
\newcommand{\beaa}{\begin{eqnarray*}}
\newcommand{\eeaa}{\end{eqnarray*}}

\newcommand{\xaxis}{Varying $\lambda$ on the $x$-axis.\xspace}

\newcommand{\bernoullitext}{The blue curve represents the $h$ function of the symmetric $\Ber\left(\frac{1}{2}\right)$.\xspace}

\makeatletter

\newcommand{\addresses}[1]{
\par {\raggedright #1
\vspace{1.4em}
\noindent\par}
}

\makeatother

\begin{document}

\title{On strict sub-Gaussianity, optimal proxy variance\\
and symmetry for bounded random variables}

\author{Julyan Arbel$^{1*}$, Olivier Marchal$^{2}$, Hien D. Nguyen$^{3}$}
\date{}

\maketitle

\addresses{$^{*}$Corresponding author, email: julyan.arbel@inria.fr.\\
$^{1}$Universit\'e Grenoble Alpes, Inria, CNRS, Grenoble INP, LJK, 38000 Grenoble, France.\\
$^{2}$Universit\'e de Lyon, CNRS UMR 5208, Universit\'e Jean Monnet, Institut Camille Jordan, 69000 Lyon, France.\\
$^{3}$Department of Mathematics and Statistics, La Trobe University, Bundoora Melbourne 3086, Victoria Australia. \\
}

{\small 
\begin{center}
    {\bf Abstract}
\end{center}
\begin{quote}
We investigate the sub-Gaussian property for almost surely bounded random variables. 
If sub-Gaussianity \textit{per se} is \textit{de facto} ensured by the bounded support of said random variables, then exciting research avenues remain open.
Among these questions is how to characterize the optimal sub-Gaussian proxy variance? Another question is how to characterize \textit{strict} sub-Gaussianity, defined by a proxy variance equal to the (standard) variance? 
We address the questions in proposing conditions based on the study of functions variations. A particular focus is given to the relationship between strict sub-Gaussianity and symmetry of the distribution. In particular, we demonstrate that symmetry is neither sufficient nor necessary for strict sub-Gaussianity. In contrast, simple necessary conditions on the one hand, and simple sufficient conditions on the other hand, for strict sub-Gaussianity are provided. 
These results are illustrated via various applications to a number of bounded random variables, including Bernoulli, beta, binomial, uniform, Kumaraswamy, and triangular distributions.
\end{quote}
}

\newpage
\section{Introduction}\label{sec:intro}

Sub-Gaussian distributions are probability distributions that have tail probabilities that are upper bounded by Gaussian tails. More specifically, a random variable $X$ with finite mean $\mu=\E[X]$ is {sub-Gaussian} if there exists $\sigma^2>0$ such that:
\begin{align}\label{eq:def}
\E[\exp(\lambda (X-\mu))]\le\exp\left(\frac{\lambda^2\sigma^2}{2}\right)\,\,\text{, for all } \lambda\in\R.
\end{align}
The constant $\sigma^2$ is called a \textit{proxy variance} and  $X$ is termed $\sigma^2$-sub-Gaussian. For a sub-Gaussian random variable $X$, the smallest  proxy variance is called the \textbf{optimal proxy variance} and is denoted $\sigma_\opt^2(X)$, or simply $\sigma_\opt^2$.
The variance always provides a lower bound on the optimal proxy variance: $\Var[X]\leq \sigma_\opt^2(X)$. When $\sigma_\opt^2(X)=\Var[X]$, $X$ is said to be \textit{strictly} sub-Gaussian.

The sub-Gaussian property is increasingly studied and used in various fields of probability and statistics, primarily due to its intricate link with concentration inequalities \citep{boucheron2013concentration,raginsky2013concentration}, transportation inequalities \citep{bobkov1999exponential,van2014probability} and PAC-Bayes inequalities \citep{catoni2007pac}. Applications include the missing mass problem \citep{mcallester2003concentration, berend2013concentration,ben2017concentration}, multi-armed bandit problems \citep{bubeck2012regret} and singular values of random matrices \citep{rudelson2010non}. 

This paper focuses on the study of almost surely bounded random variables, where Bernoulli, beta, binomial, Kumaraswamy \citep{jones2009kumaraswamy} or triangular \citep{kotz2004beyond} distributions are taken as standard and common examples. 
If sub-Gaussianity \textit{per se} is \textit{de facto} ensured because the support of said random variables is bounded, then exciting research avenues remain open in the area. 
Among these questions are (a) how to obtain the optimal sub-Gaussian proxy variance, and (b) how to characterize \textit{strict} sub-Gaussianity? 

Regarding question (a), we propose general conditions characterizing the  optimal sub-Gaussian proxy variance, thus generalizing previous work \citep{marchal2017sub} that was tailored to the beta and Dirichlet distributions. Several techniques based on studying variations of functions are proposed. In illustrating our results with the Bernoulli distribution, we prove as a by-product of Proposition~\ref{prop:bernoulli} the uniqueness of a global maximum of a function that was observed by \cite{berend2013concentration}  ``as an intriguing open problem''.

As for  question (b), it turns out that the \textit{symmetry} of the distribution plays a crucial role. By symmetry, we mean symmetry with respect to the mean $\mu=\E[X]$. That is, we say that $X$ is symmetrically distributed if $X$ and $2\mu-X$ have the same distribution. Thus, if $X$ has a density, this means that the density is symmetric with respect to $\mu$. 
A simple, and remarkable, equivalence holds for most of the standard bounded random variables. 

\begin{proposition}
    \label{prop:symmetry-strict}
	Let $X$ be a Bernoulli, beta, binomial, Kumaraswamy or triangular random variable. Then,
	\begin{center}
	$X$ is symmetric $\Longleftrightarrow$ 	$X$ is strictly sub-Gaussian.
	\end{center}
\end{proposition}
The result is known for the beta distribution \citep{marchal2017sub}. In this article, we provide proofs for the Bernoulli, binomial, Kumaraswamy and triangular distributions. 

From Proposition \ref{prop:symmetry-strict}, it may be tempting to conjecture that the equivalence holds true for \textit{any} random variable having a bounded support. However, we establish that this is not the case. This was actually one of the starting points for the present work. 
More precisely, we shall provide a proof of the following result.

\begin{proposition}
    \label{prop:symmetry-not-NSC}
    Symmetry of $X$ is neither
    \begin{itemize}
        \item[(i)] a sufficient condition, nor
        \item[(ii)] a necessary condition,
    \end{itemize} 
    for the strict sub-Gaussian property.
\end{proposition}
The proof of this result is presented in Section~\ref{sec:strict-symmetry}, where  we demonstrate that \textit{(i)} there exists simple symmetric mixtures of distributions (e.g., a two-components mixture of beta distribution and a three-components mixture distribution of Dirac masses) which are not strictly sub-Gaussian, and that \textit{(ii)} there exists an asymmetric three-components mixture of Dirac masses which is strictly sub-Gaussian.  

Before delving into detailing the strict sub-Gaussianity property in Section~\ref{sec:strict}, we first investigate some conditions that characterize the optimal proxy variance $\sigma_\opt^2$, in Section~\ref{sec:optimal-proxy}. 
The results of Sections \ref{sec:optimal-proxy} and \ref{sec:strict} are then illustrated on a number of standard random variables on bounded supports, in Section~\ref{sec:illustrations}. Technical results are presented in  Appendix~\ref{sec:appendix}.

\newpage

\section[Characterizations of the optimal proxy variance]{Characterizations of the optimal proxy variance $\sigma_\opt^2$}\label{sec:optimal-proxy}

Let $X$ be an almost surely bounded random variable with mean $\mu=\E[X]$. Then, $X$ is sub-Gaussian and satisfies Definition~\ref{eq:def} for some $\sigma^2>0$.

An equivalent definition is that 
$$\forall\, \lambda\in \mathbb{R}\,:\, \sigma^2\geq \frac{2}{\lambda^2}\mathcal{K}(\lambda),$$
where the function $\mathcal{K}$, defined on $\R$ by: $\mathcal{K}(\lambda)=\ln \E[\exp(\lambda[X-\mu])]$, corresponds to the cumulants generating function of $X-\mu$. 
Thus the optimal proxy variance $\sigma_\opt^2$ can be defined as the supremum
\begin{equation}\label{eq:supremum}
  \sigma_\opt^2=\sup_{\lambda\in \mathbb{R}}\text{ } \frac{2}{\lambda^2}\mathcal{K}(\lambda).
\end{equation}
If $X$ is almost surely bounded, then this supremum is attained, see Lemma~\ref{lem:sup-equals-max} for details. 
Note that the function $h$, defined on $\R$ by
\begin{equation}\label{eq:h-def}
h(\lambda)=\frac{2}{\lambda^2}\mathcal{K}(\lambda),
\end{equation}
is continuous at $\lambda=0$,  since a standard series expansion demonstrates that:
\begin{equation}\label{eq:h(0)=Var(X)}
    h(\lambda)\overset{\lambda \to 0}{=}\Var[X]+o(1).
\end{equation}
Moreover, $h$ may never vanish. In fact, since the logarithm function is strictly concave, Jensen's inequality implies that for any $\lambda\in\mathbb{R}$,
\begin{align}\label{eq:h-positive}
    h(\lambda)=\frac{2}{\lambda^2}\ln\E[\edr^{\lambda(X-\mu )}]>\frac{2}{\lambda^2}\E[\ln\edr^{\lambda(X-\mu )}]=0.
\end{align}
Equation~\eqref{eq:h(0)=Var(X)}  also explains directly why $\sigma_\opt^2\geq \Var[X]$, since the variance is the value of the right-hand side (r.h.s.) function at $\lambda=0$ and thus the maximum is always greater or equal to it. 
We therefore have the following result.%
\begin{proposition}[Characterization of $\sigma_\opt^2$ by $h$]\label{prop:general-proxy-variance}
The optimal proxy variance is given by:    
\begin{equation}\label{eq:sigma_opt-as-max-of-h}
    \sigma_\opt^2=\max_{\lambda\in \mathbb{R}}\text{ }h(\lambda)=\max_{\lambda\in \mathbb{R}}\text{ } \frac{2}{\lambda^2}\mathcal{K}(\lambda).
\end{equation}
\end{proposition}

We may now present a necessary (but not always sufficient) system of equations for $\sigma_\opt^2$. Indeed, since the maximum is achieved at a finite point, then this point must necessarily be a zero of the derivative of $h$, if $h$ is
differentiable (we will denote by $\mathcal{D}^k$ the space of functions that are $k$ times differentiable on $\R$ and by $\mathcal{C}^k$ the space of functions that are $k$ times differentiable on $\R$ and for which the $k^{\text{th}}$ derivative is continuous on $\R$).

Thus, we obtain the following corollary.
\begin{corollary}[Necessary condition for $\sigma_\opt^2$, with respect to $h$]\label{cor:nece-cond-h-maximum}
Let $\sigma_\opt^2$ be the optimal proxy variance, and assume that $h$ and $\mathcal{K}$ are $\mathcal{D}^1$. 
Then there exists a finite $\lambda_0$, such that
	\begin{equation}
    \sigma_\opt^2=h(\lambda_0) \text{ and } h'(\lambda_0)=0,
\end{equation}
which is equivalent to
\begin{equation}
    \sigma_\opt^2=\frac{2}{\lambda_0^2}\mathcal{K}(\lambda_0) \text{ and } \lambda_0\mathcal{K}'(\lambda_0)=2\mathcal{K}(\lambda_0),
\end{equation}
using only the centered cumulants generating function $\mathcal{K}$.
\end{corollary}
In practice, the previous set of equations has to used with caution, since there may be more than one solution to the second equation involving the derivative of $h$ (or that of $\mathcal{K}$), and a global maximizer is required to be picked among the stationary points, instead of a minimizer or a local maximizer. 
On a case-by-case basis, the following approach based on ordinary differential equations (ODEs), satisfied by  $h$, can be used to demonstrate that it has a unique global maximum. 
\begin{proposition}\label{prop:ODEs}
 If the function  $h$ is $\mathcal{C}^2$, then it is the unique solution of the ordinary differential equations:
\begin{equation}\label{eq:ODE-first-order}
    h'(\lambda)+\frac{2}{\lambda}h(\lambda)=\frac{2}{\lambda^2}\mathcal{K}'(\lambda) \text{ with } h(0)=\Var[X],
\end{equation}
or
\begin{equation}\label{eq:ODE-second-order}
    h''(\lambda)+\frac{3}{\lambda}h'(\lambda)=\frac{2}{\lambda}\left(\frac{\mathcal{K}'(\lambda)}{\lambda}\right)' \text{ with } h(0)=\Var[X] \text{  and } h'(0)=\frac{1}{3}\E[(X-\mu)^3]. 
\end{equation}   
\end{proposition}
\begin{proof}
    The result is directly obtained by differentiating  $h$ and via standard analysis theorems. 
\end{proof}
\begin{remark}
For cases such as the Bernoulli and uniform distributions, we may  prove that the r.h.s. of~\eqref{eq:ODE-second-order} is strictly negative on $\mathbb{R}^*:=\R\setminus\{0\}$. This implies that if $\lambda_0$ is extremal (i.e.,  $h'(\lambda_0)=0$), then it satisfies $h''(\lambda_0)<0$ so that it is a local maximum. This implies that $h$ has no local minimum and thus may only have one critical point which is necessarily the unique global maximum. 
\end{remark}

We conclude this section with another possible methodology for  deriving a necessary and sufficient condition for $\sigma_\opt^2$. To this end, the problem needs to be addressed from a different point of view, by studying the difference of the terms of Definition~\ref{eq:def}:
\begin{equation}\label{eq:Delta}
\Delta\,:\, (\sigma^2,\lambda)\in\mathbb{R}_+^*\times \mathbb{R}\mapsto \exp\left(\frac{\lambda^2\sigma^2}{2}\right)-\E[\exp(\lambda[X-\mu])].
\end{equation}
\begin{proposition}[Characterization of $\sigma^2_\opt$, with respect to $\Delta$]\label{prop:sigma_opt-NSC}
If $\Delta$ is  $\mathcal{C}^1$, then the optimal proxy variance is characterized by:    
\begin{equation}\label{eq:th:sigma_opt-NSC}
\lambda\mapsto \Delta(\sigma^2_\opt,\lambda)\geq 0 \text{ and } \exists\, \lambda_0\in \mathbb{R}\text{, such that } \Delta(\sigma^2_\opt,\lambda_0)=0 \text{ and } \partial_\lambda\Delta(\sigma^2_\opt,\lambda_0)=0.
\end{equation}
\end{proposition}
\begin{proof}
	See Section~\ref{sec:proof-Delta}, in Appendix.
\end{proof}
This proof technique was used by \cite{marchal2017sub} for obtaining the optimal proxy variance of the beta and Dirichlet distributions. 
However we find more convenient to use the conditions stated in Proposition~\ref{prop:ODEs} using the function $h$ to address the issues presented in this article, except for the triangular distribution in Section~\ref{sec:triangular} where this  method is employed for a numerical evaluation of $\sigma_\opt^2$. 
\begin{remark}
In general, we would like to remove the condition: $\lambda\mapsto \Delta(\sigma^2,\lambda)\geq 0$ on the r.h.s. of Proposition~\ref{prop:sigma_opt-NSC}, in order to have a simpler (and local) characterization of the optimal proxy variance, as a solution of \eqref{eq:th:sigma_opt-NSC}. However, this is not possible, since we may not exclude that there exists a value $\sigma^2<\sigma^2_\opt$ for which $\Delta(\sigma^2,\lambda)$ presents a double zero $\lambda_0$ where locally it remains non-negative but at the same time a whole interval far from $\lambda_0$ where it would be strictly negative. 
\end{remark}

\section{On strict sub-Gaussianity}\label{sec:strict}

\subsection{Conditions based on the cumulants}

Strict sub-Gaussianity is fulfilled when the optimal proxy variance equals the variance. In view of Equation~\eqref{eq:h(0)=Var(X)}, Proposition~\ref{prop:general-proxy-variance} can be rewritten as the following corollary in order to characterize the strict sub-Gaussianity property.
\begin{corollary}[Corollary of Proposition~\ref{prop:general-proxy-variance}]\label{cor:NSC-strict-h-max-0}
    A distribution is strictly sub-Gaussian if and only if the maximum of function $h$, defined in~\eqref{eq:h-def}, is attained in zero (and is automatically equal to $\Var[X]$). That is: 
\begin{equation}
 \max_{\lambda\in \mathbb{R}} \text{ } h(\lambda)=h(0)=\Var[X].
\end{equation}
\end{corollary}
This characterization provides necessary conditions, based on cumulants, that are required for strict sub-Gaussianity to hold. 
\begin{proposition}[Necessary conditions based on cumulants]\label{prop:skewness}
If $X$ is strictly  sub-Gaussian, then the $3^{\text{rd}}$ and $4^{\text{th}}$ cumulants of $X$ must satisfy
\begin{align}
    \kappa_3 &= \E[(X-\E[X])^3]=0\text{, and }\label{eq:kappa_3}\\
    \kappa_4 &= \E[(X-\E[X])^4]-3\Var[X]^2\leq0.\label{eq:kappa_4}
\end{align}
\end{proposition}
\begin{proof}
By definition of the cumulant generating function $\mathcal{K}(\lambda)$ of $X-\mu$, 
\begin{equation}\label{eq:cumulant-gen-fun-def}
	\mathcal{K}(\lambda)=\sum_{i=1}^\infty \kappa_i\frac{\lambda^i}{i!},
\end{equation}
where $\kappa_i$ are the cumulants of $X-\mu$. Since $\kappa_1=\mu-\mu=0$ and $\kappa_2=\Var[X]$, and using values for the third and fourth cumulants given in~\eqref{eq:kappa_3} and~\eqref{eq:kappa_4}, we may write (locally around $\lambda\to 0$):
\begin{equation}\label{eq:h-taylor}
    h(\lambda)=\Var[X]+ \E[(X-\mu)^3]\frac{\lambda}{3}+\left(\E[(X-\mu)^4]-3\Var[X]^2\right)\frac{\lambda^2}{12}+O(\lambda^3).
\end{equation}

Therefore if $\E[(X-\mu)^3]\neq 0$, the maximum of $h(\lambda)$ cannot be $h(0)$ and thus strict sub-Gaussianity cannot be achieved. We conclude the proof by noting that if $\E[(X-\mu)^3]= 0$, we have the fact that $\lambda=0$ can be a local maximum, only if $\E[(X-\mu)^4]\leq 3\Var[X]^2$. 
\end{proof}
Condition~\eqref{eq:kappa_3} requires that the third centered moment is zero and Condition~\eqref{eq:kappa_4} imposes a relation between the second and fourth centered moments. Note that the latter condition  can be compactly formulated via an alternative condition on the kurtosis of $X$:
\begin{align*}
    \text{Kurt}[X] = \frac{\E[(X-\E[X])^4]}{\E[(X-\E[X])^2]^2}\leq 3.
\end{align*}
More specifically, sub-Gaussianity requires that the random variable has kurtosis less than or equal to three, which is the kurtosis of a standard Gaussian random variable. Such distributions are referred to as \textit{platycurtic}. The fourth cumulant defined in~\eqref{eq:kappa_4} is also termed \textit{excess kurtosis}. Thus, strict sub-Gaussianity requires negative excess kurtosis.

When the above necessary conditions~\eqref{eq:kappa_3} and~\eqref{eq:kappa_4} hold, we are not able to obtain simple additional necessary conditions on the next cumulants. In particular, note that strict sub-Gaussianity \textit{does not} imply symmetry (i.e.,  $\E[(X-\E[X])^{2j +1}]=0$, for any $j \geq 0$), as will be discussed in the next section. 

In contrast, more can be said when the distribution is symmetric. In fact, in the symmetric case, the moments of odd order are zero, and a simple sufficient condition can be readily obtained by comparing the Taylor expansions at $\lambda=0$ of both terms of inequality~\eqref{eq:def}, as stated in the following proposition. 

\begin{proposition}[Sufficient condition based on moments]\label{prop:symmetric-moments}
If $X$ is symmetric with respect to its mean $\mu=\E[X]$, then a sufficient condition for $X$ to be strictly sub-Gaussian can be stated in terms of all its even moments. That is, for $X$ to be strictly sub-Gaussian, it is sufficient that
\begin{align}\label{eq:sufficient-symmetric-moments}
    \forall j\geq 2,\quad 
    \frac{\E[(X-\mu)^{2j}]}{(2j)!} \leq \frac{(\Var[X])^j}{2^{j}j!}
\end{align}
holds.
\end{proposition}
\begin{proof}
The proof is based on series expansions at $\lambda=0$ of both terms of inequality~\eqref{eq:def}, when the proxy variance $\sigma^2$ is set to the variance $\Var[X]$. Namely:
\begin{align}\label{eq:AlphaAlphaCase}
\E[\exp(\lambda X)] = \sum_{j=0}^{\infty} \E\left[X^{2j}\right]\frac{\lambda^{2j}}{(2j)!},\quad \text{and} \quad 
\exp\left(\frac{\lambda^2\Var[X]}{2}\right)=\sum_{j=0}^{\infty} \frac{(\Var[X])^j}{2^{j}}\frac{\lambda^{2j}}{j!},
\end{align} 
when compared term-by-term, leads to inequality~\eqref{eq:def}, under assumption~\eqref{eq:sufficient-symmetric-moments}. 
Note that inequality~\eqref{eq:sufficient-symmetric-moments} needs be checked only for $j\geq 2$, as it trivially holds for $j=0,1$.
\end{proof}

This technique was used by \citet{marchal2017sub} (Section 2.2)  for showing  that a (symmetric) $\Beta(\alpha,\alpha)$ random variable is strictly sub-Gaussian. We also use it to address the cases of Bernoulli and binomial, and triangular distributions in Section~\ref{sec:illustrations}.

\subsection{Link with symmetry}\label{sec:strict-symmetry}

The relationship between strict sub-Gaussianity and symmetry was discussed in the Introduction. Here, we provide a proof of Proposition~\ref{prop:symmetry-not-NSC}, while the proof of Proposition~\ref{prop:symmetry-strict} is deferred to Section~\ref{sec:illustrations}.

\subsubsection{Symmetry is neither a sufficient condition\ldots\label{sec:sym-not-strict}}
Simple symmetric distributions which break the necessary condition of negative excess kurtosis can easily be constructed by hand. One such construction is by means of mixture of Dirac masses. First, consider the discrete random variable 
\begin{equation}\label{eq:sym-not-strict}
	X\sim \frac{\eta}{2} (\delta_{-1}+\delta_{1})+(1-\eta)\delta_{0},
\end{equation}
which is a three-component mixture of Dirac masses at locations $-1$, $0$ and $1$, with $\eta \in [0,1]$. It is symmetric, by construction, and its excess kurtosis equals
\begin{equation}
	\kappa_4 = \E[X^4]-3\Var[X]^2 = \eta - 3\eta^2 = \eta(1-3\eta),
\end{equation}
which is \textit{strictly positive} for all values $\eta \in\left(0, \frac{1}{3}\right)$, hence $X$ is not strictly sub-Gaussian for these values by virtue of Proposition~\ref{prop:skewness}. On the other hand when $\eta \to 1$, the distribution of $X$ degenerates to that of the so-called Rademacher random variable, which leads to the least possible excess kurtosis of $-2$.

Similar counter-examples to the sufficientness of symmetry can be built in the form of mixtures of two symmetric beta variables:
\begin{align*}
    X\sim\eta \Beta(\alpha,\alpha)+(1-\eta)\Beta(\beta,\beta),
\end{align*}
for $\eta \in (0,1)$ and $\alpha,\beta>0$. For any value of  $\eta \in (0,1)$, values for  $\alpha,\beta$ leading to positive excess kurtosis can be obtained. For instance, we may set $(\eta,\alpha,\beta)=(0.1,1.5,9)$, to obtain the excess kurtosis $\kappa_4 \approx 1.1\times 10^{-4}$.

\subsubsection{\ldots nor a necessary condition for strict sub-Gaussianity\label{sec:asym-strict}}
Although most typical bounded random variables that are strictly sub-Gaussian are symmetric (see, e.g., Proposition~\ref{prop:symmetry-strict}), the symmetry of the distributions of such variables is not a necessary condition for strict sub-Gaussianity. Examples of such distributions include mixtures of Dirac masses.  For example,
\begin{align}\label{eq:not_necessary}
    X\sim \sum_{i=1}^3 p_i \delta_{x_i} \text{ with } \sum_{i=1}^3 p_i=1
\end{align}
with $(x_1,x_2,x_3)=\left(-2,-\frac{1}{2},\frac{5}{4}\right)$ and $(p_1,p_2,p_3)=\left(\frac{1}{13},\frac{4}{7},\frac{32}{91}\right)$. The function $h$ for the random variable characterized by \eqref{eq:not_necessary} is plotted in Figure~\ref{fig:not_necessary}. Note that it attains its maximum in $\lambda=0$.

\begin{figure}[!ht]
    \centering
    \begin{subfigure}{0.49\textwidth}
            \centering
            \includegraphics[width=\textwidth]{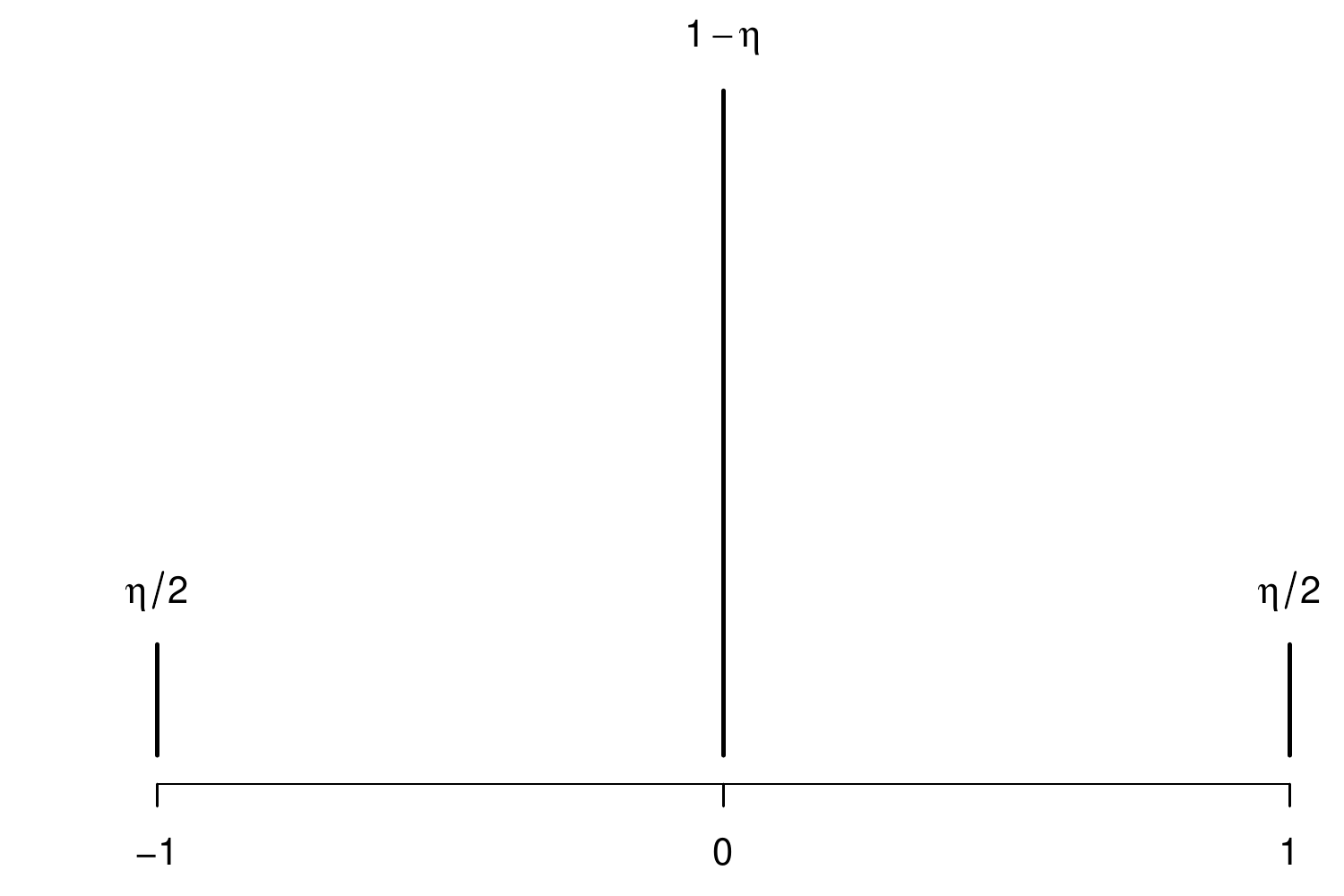}
    \subcaption{Distribution~\eqref{eq:sym-not-strict}: symmetric but not strictly sub-Gaussian when $\eta\in\left(0,\frac{1}{3}\right)$}
    \label{fig:dirac-sym-not-strict}
    \end{subfigure}
    \begin{subfigure}{0.49\textwidth}
            \centering
            \includegraphics[width=\textwidth]{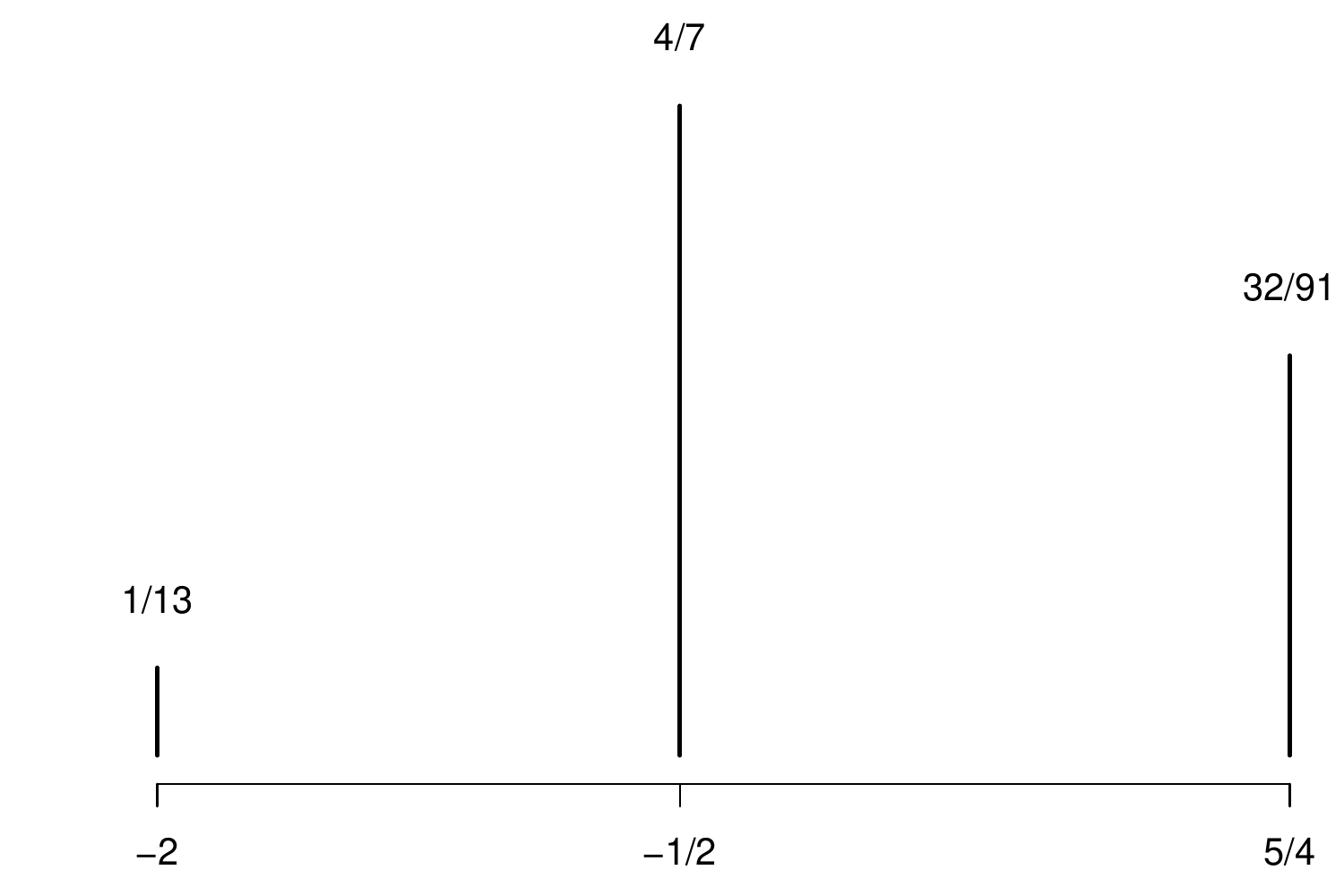}
    \subcaption{Distribution~\eqref{eq:not_necessary}: asymmetric but  strictly sub-Gaussian}
    \label{fig:not_necessary}
    \end{subfigure}\\
    \begin{subfigure}{0.49\textwidth}
            \centering
            \includegraphics[width=\textwidth]{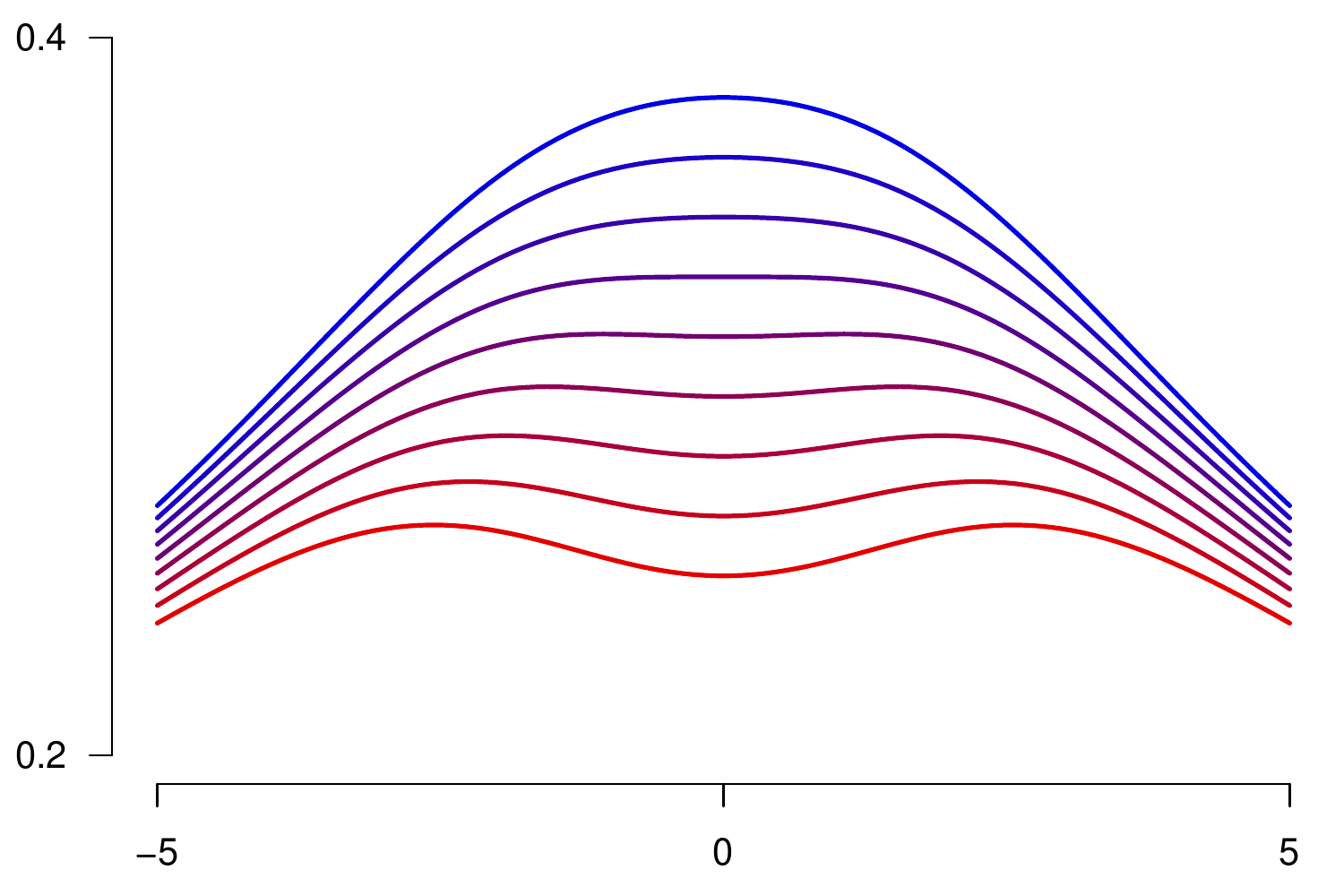}
    \subcaption{Function $h$ for~\eqref{eq:sym-not-strict} with $\eta\in[0.25,0.4]$}
    \label{fig:h-sym-not-strict}
    \end{subfigure}
    \begin{subfigure}{0.49\textwidth}
            \centering
            \includegraphics[width=\textwidth]{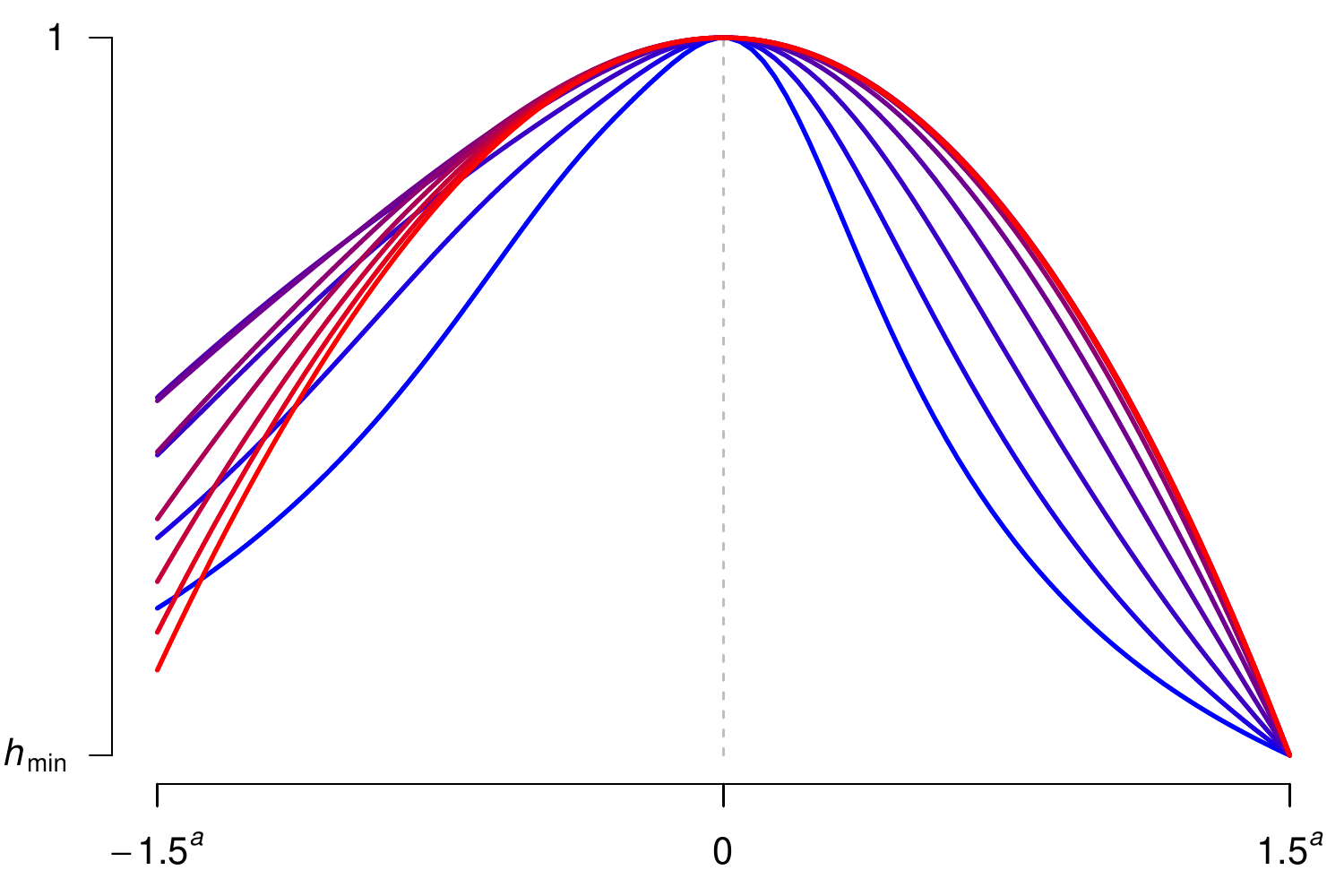}
    \subcaption{Function $h$ for~\eqref{eq:not_necessary}, zoomed-in for red curves}
    \label{fig:h-asym-strict}
    \end{subfigure}
    \caption{Illustration of the mixtures of Dirac masses, proving Proposition~\ref{prop:symmetry-not-NSC} described in Sections~\ref{sec:sym-not-strict} (a,c) and~\ref{sec:asym-strict} (b,d). In (c), $\eta$ varies from $0.25$ (red curve, maximum not at zero) to $0.4$ (blue curve, maximum at zero). In (d), we illustrate the function $h$ with different zooming scales (around $\lambda=0$): from $[-1.5^5,1.5^5]$ (blue curve, maximum zoom-out) to $[-1.5^{-5},1.5^{-5}]$ (red curve, maximum zoom-in), with an adapted $y$-scale, showing that the maximum is attained in zero.}
\end{figure}

\pagebreak

\section{Results and applications to standard distributions\label{sec:illustrations}}

\subsection{Bernoulli and binomial distributions\label{sec:bernoulli}}

Consider a Bernoulli random variable, $X\sim \Ber(\mu)$ with $\mu\in(0,1)$ and a binomial random variable, $Y\sim \Bin(n,\mu)$ which can be obtained as the sum of $n$ independent  $\Ber(\mu)$
random variables, $n$ a positive integer. 

\begin{proof}[Proof of Proposition~\ref{prop:symmetry-strict} for the Bernoulli and binomial distributions]$ $\newline
Starting with the Bernoulli: the third cumulant is equal to
$$\E[(X-\E[X])^3]=2\mu(1-\mu)\left(\frac{1}{2}-\mu\right),$$
thus, by virtue of Proposition~\ref{prop:skewness}, a non-degenerate Bernoulli random variable may only be strictly sub-Gaussian when $\mu=\frac{1}{2}$. That is, when it is symmetric.

Conversely,  verifying the sufficient condition for
the symmetric Bernoulli distribution $\Ber\left(1/2\right)$ is equivalent to assessing the condition for the Rademacher distribution instead. That is, the distribution of random
variable $X$, where the events $X=-1$ and $X=1$ have equal probability
\[
\mathbb{P}\left[X=-1\right]=\mathbb{P}\left[X=1\right]=\frac{1}{2}\text{.}
\]
Since $X^2=1$, the variance of $X$ and all of its even moments are $\Var[X]=\E[X^{2j}]=1$. 
Therefore, to verify the sufficient condition  of Proposition~\ref{prop:symmetric-moments}, we are required
to demonstrate that 
\[
\left(2j\right)!\ge2^{j}j!,
\]
for each $j\ge2$, which follows from the expansion
\[
\left(2j\right)!=\underset{j}{\underbrace{2j\times\dots\times\left(j+1\right)}}\times j! \ge2^{j}j!.
\]
Thus, we have verification of the sufficient
condition for the Rademacher distribution and hence the symmetric
Bernoulli distribution, as a consequence.

Turning to the binomial distribution, we observe that the optimal proxy variance of a sum of i.i.d. (independent and identically distributed) variables is the sum of the optimal proxy variances. Thus, we immediately obtain the result that 
\begin{equation}
    \sigma_\opt^2[\Bin(n,\mu )]=n\sigma_\opt^2[\Ber(\mu )].
\end{equation}
In particular, $X\sim \Bin(n,\mu )$ is strictly sub-Gaussian if and only if $\mu =\frac{1}{2}$.
\end{proof}

We now turn to the optimal proxy variance of a Bernoulli, which has the form
\begin{equation}
	\sigma_\opt^2 = \frac{\frac{1}{2}-\mu }{\ln\left(\frac{1}{\mu} -1\right)}. 
\end{equation}
This fact is known via Theorem 2.1 and Theorem 3.1 of \cite{buldygin2013binary}; see also the discussion in the introduction of \cite{marchal2017sub}. 
Here, we focus on a rather different approach, based on function $h$ and Corollary~\ref{cor:nece-cond-h-maximum}, where
\begin{equation}
h(\lambda)=\frac{2}{\lambda^2}\left(\ln[\mu  \edr^\lambda+(1-\mu )]- \mu \lambda\right).
\end{equation}
Note that the study of the variations of $h$ is observed by \cite{berend2013concentration} (cf. their function $g$; Equation (2.1)). However, a formal proof that $h$ has a single global maximum is left ``as an intriguing open problem'' by \citeauthor{berend2013concentration}. This is stated in the next proposition, and formally proved, below. An illustration of this result is presented in Figure~\ref{fig:bernoulli}.

\begin{proposition}\label{prop:bernoulli}
	If $X\sim \Ber(\mu )$, then the function 
\begin{equation*}
    h: \lambda \mapsto  \frac{2}{\lambda^2}\left(\ln(\mu  \edr^\lambda+1-\mu)- \mu \lambda\right)
\end{equation*} admits a unique critical point which is a global maximum. The global maximizer is obtained at $\lambda_0=2\ln\frac{1-\mu}{\mu}$, which leads to the optimal proxy variance of form
$$\sigma_{\opt}^2=h(\lambda_0)=\frac{\frac{1}{2}-\mu }{\ln\left(\frac{1}{\mu} -1\right)}.$$ 
\end{proposition}
\begin{proof}
Let us first prove that $h$ admits a unique critical point, which is a global maximum, by using Proposition~\ref{prop:ODEs} and the remark that follows.  ODEs~\eqref{eq:ODE-first-order} and~\eqref{eq:ODE-second-order} are respectively
\begin{equation}
    h'(\lambda)+\frac{2}{\lambda}h(\lambda)\coloneqq r(\lambda)=\frac{2\mu (1-\mu )(\edr^\lambda-1)}{\lambda^2(\mu \edr^\lambda+1-\mu )}\,,h(0)=\mu (1-\mu)
\end{equation}
and
\begin{equation}
    h''(\lambda)+\frac{3}{\lambda}h'(\lambda)=r'(\lambda)+\frac{r(\lambda)}{\lambda}=\frac{2\mu (1-\mu )}{\lambda^3(\mu \edr^\lambda+1-\mu )}\left((\lambda+2\mu -1)\edr^\lambda-\mu \edr^{2\lambda}+1-\mu \right),
\end{equation}
with $ h'(0)=\frac{\mu (1-\mu )(2\mu -1)}{3}$. 
Let us denote $g(\lambda)=(\lambda+2\mu -1)\edr^\lambda-\mu \edr^{2\lambda}+1-\mu $, $u=\edr^\lambda>0$ and $G(u)=u\ln(u)+(2\mu -1)u-\mu u^2+1-\mu $. We have $G''(u)=\frac{1}{u}-2\mu $ so that $G''$ is positive on $\left[0,\frac{1}{2\mu}\right]$ and negative on $[\frac{1}{2\mu},+\infty)$. Since $G'\left(\frac{1}{2\mu}\right)=-\ln(2\mu )+2\mu (1-2\mu )<0$, for $\mu >\frac{1}{2}$, we have the fact that $G'$ is always strictly negative on $\mathbb{R}_+$. Thus, $G$ is a strictly decreasing function of $u$ and hence $g$ is also a strictly decreasing function of $\lambda$. Note that $g(0)=0$, so that $g$ is positive on $\mathbb{R}_-$ and negative on $\mathbb{R}_+$. Thus, $r'(\lambda)+\frac{r(\lambda)}{\lambda}$ is always strictly negative (there is a factor $\lambda^3$ in the denominator that changes sign at $\lambda=0$). We conclude that a point $\lambda_2\neq 0$, where $h'(\lambda_2)=0$, always satisfies $h''(\lambda_2)<0$ and therefore it is always a local maximum of $h$. 

By differentiating $h$, we observe that the global maximizer of $h$ is obtained as the unique solution (in $\lambda_0$) of the equation
\begin{equation*}
     2(\mu \edr^{\lambda_0}+1-\mu )\ln(\mu \edr^{\lambda_0}+1-\mu )-\mu \lambda_0\left((1+\mu )\edr^{\lambda_0}+1-\mu \right)=0.
\end{equation*}
It is easy to verify that $\lambda_0=2\ln\left(\frac{1-\mu}{\mu}\right)$, and that this leads to the optimal proxy variance as stated.
\end{proof}

\begin{figure}[!ht]
    \centering
    \includegraphics[width=.5 \textwidth]{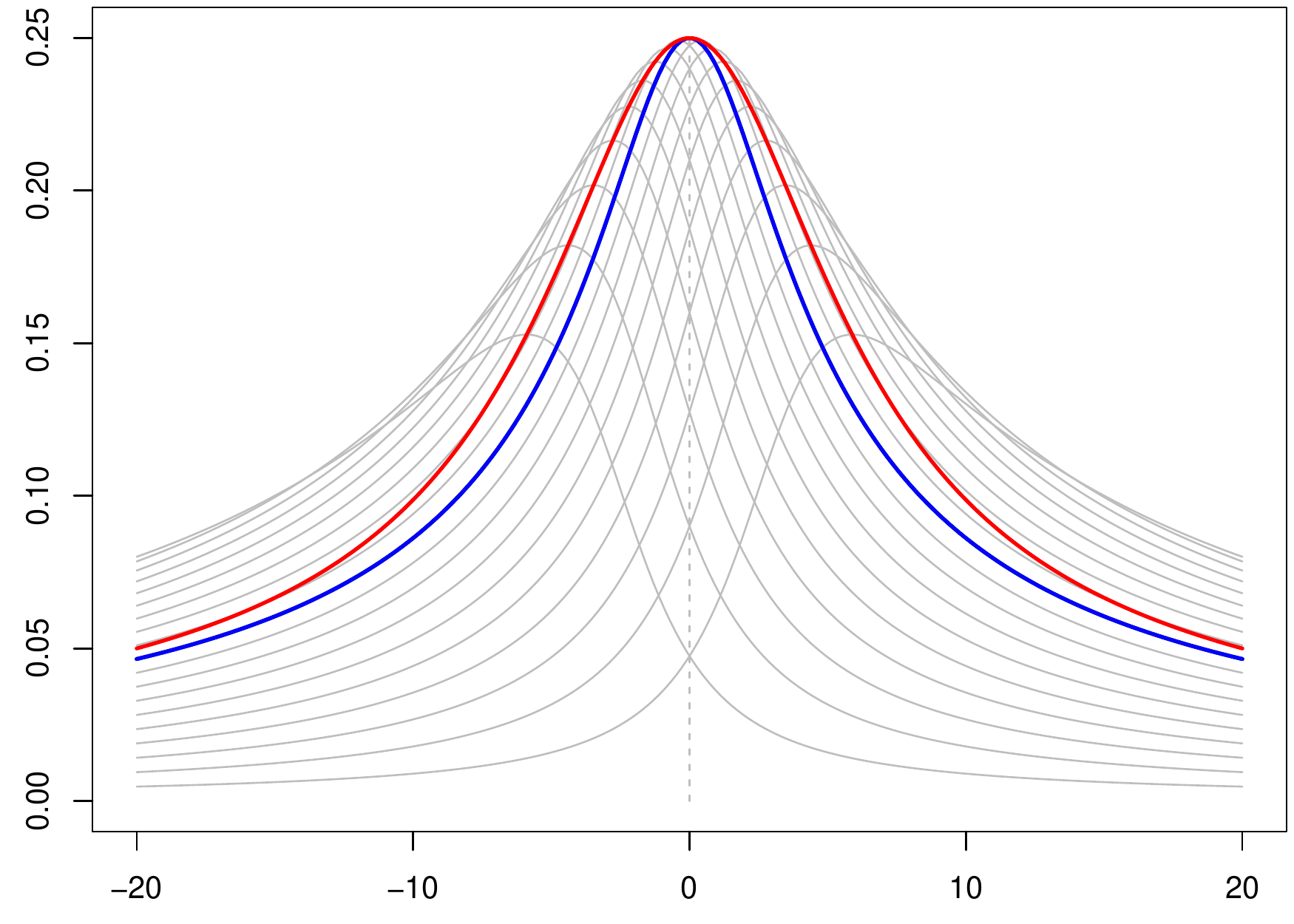}
    \caption{    \label{fig:bernoulli}
    Illustration of the function $h$ for $\Ber(\mu)$ random variables 
    with $\mu$ varying from $0.05$ to $0.95$  (gray curves), and of the particular case of $\Ber\left(\frac{1}{2}\right)$ (blue curve). The maximum of $h$ is realized at zero, only in the $\Ber\left(\frac{1}{2}\right)$ case, which corresponds to a symmetric distribution. Maxima of $h$, corresponding to the optimal proxy variance, are depicted by the red curve. \xaxis 
    }
\end{figure}
\pagebreak

\subsection{Triangular distribution\label{sec:triangular}}

We say that $X\sim\text{Tri}(a,b)$ is a triangular random variable on $(-a,b)$, for any $a,b>0$, if it is characterized by a density equal to 
\begin{align}
    f(x;a,b) = \left\{
    \begin{array}{ll}
          \frac{2}{a(a+b)}(x+a) &\text{ if } -a<x<0,\\
          \frac{2}{b(a+b)}(b-x) &\text{ if } 0<x<b.
    \end{array}
    \right.
\end{align}
See \citet{kotz2004beyond} for details and properties of such distributions. A recent review of research developments regarding the triangular distribution appears in \cite{nguyen2017triangular}.

\begin{proof}[Proof of Proposition~\ref{prop:symmetry-strict} for the triangular distribution]$ $\newline
The third cumulant is equal to
$$\E[(X-\E[X])^3]=\frac{(b-a)(2a+b)(2b+a)}{270},$$
so by virtue of Proposition~\ref{prop:skewness}, a triangular random variable may only be strictly sub-Gaussian when $a=b$. That is when it is symmetric. 

Conversely, when the distribution is symmetric with $a=b$, we can easily express the moments of even order in the form
$$\E[X^{2j}]=\frac{2a^{2j}}{(2j+1)(2j+2)},$$
so that the sufficient moment condition of Proposition~\ref{prop:symmetric-moments} is equivalent to
$$ \frac{\E[X^{2j}]}{(2j)!}\leq \frac{(\Var[X])^j}{2^j j!} \,\Leftrightarrow\, \frac{2}{(2j+2)!}\leq \frac{1}{12^j j!}.$$
In other words, the only remaining inequality is to show that
$$12^j j!\leq \frac{(2j+2)!}{2} \,,\,\forall j\geq 1.$$
The result is true for $j\in\{1,2,3,4\}$ by direct computation. We then make the decomposition:
$$\frac{(2j+2)!}{2}=(j+1)2^j \prod_{i=1}^j(2i+1)\geq (j+1)2^j 3\times 5\times \prod_{i=3}^j 7=\frac{15}{49}(j+1)2^j 7^j,$$
which yields
$$12^j j!\leq \frac{(2j+2)!}{2}\,\Leftrightarrow \, j+1\geq \frac{15}{49} \, \Leftrightarrow \, j\geq 3,$$
thus verifying the sufficient condition in the symmetric case.
\end{proof}

For the general case, we first observe that
$$\E[\edr^{{\lambda}(X-\E[X])}]=\frac{2}{ab(a+b){\lambda}^2}\left[a \edr^{\frac{{\lambda}(a+2b)}{3}}+b \edr^{\frac{-{\lambda}(b+2a)}{3}}-(a+b)\edr^{\frac{{\lambda}(a-b)}{3}}\right].$$
We further observe, numerically, that the difference $\Delta$ introduced in~\eqref{eq:Delta} admits a unique minimum, and also that the $h$ function~\eqref{eq:h-def} admits a unique (global) maximum; see Figure~\ref{fig:tri}. The optimal proxy variance can be obtained numerically, by minimizing~\eqref{eq:Delta}, as detailed in Proposition~\ref{prop:sigma_opt-NSC}.

\begin{figure}[!ht]
    \centering
    \begin{subfigure}{0.49\textwidth}
            \centering
            \includegraphics[width=\textwidth]{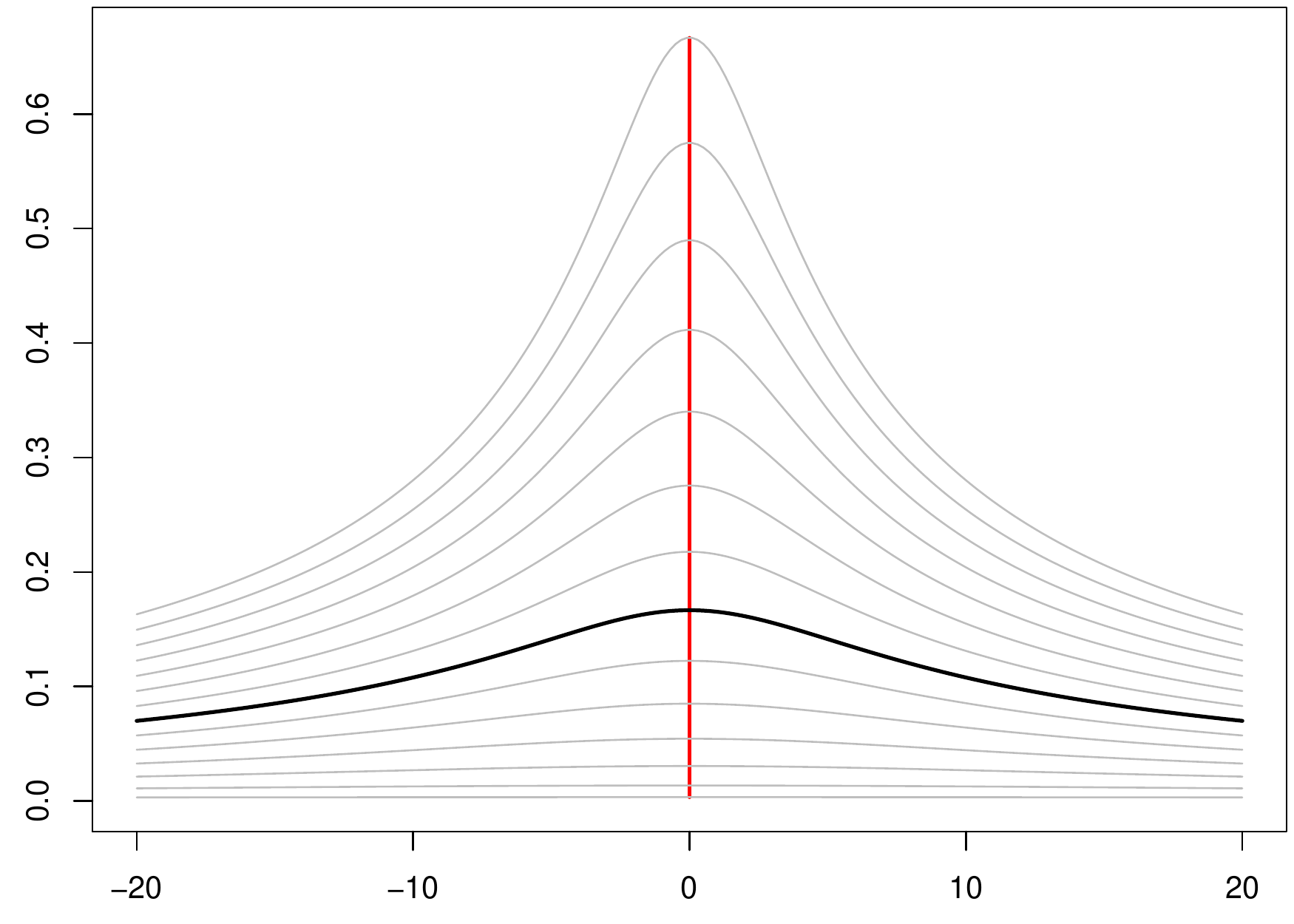}
    \subcaption{Symmetric $\Tri(a,a)$}
    \label{fig:tri_sym}
    \end{subfigure}
    \begin{subfigure}{0.49\textwidth}
            \centering
            \includegraphics[width=\textwidth]{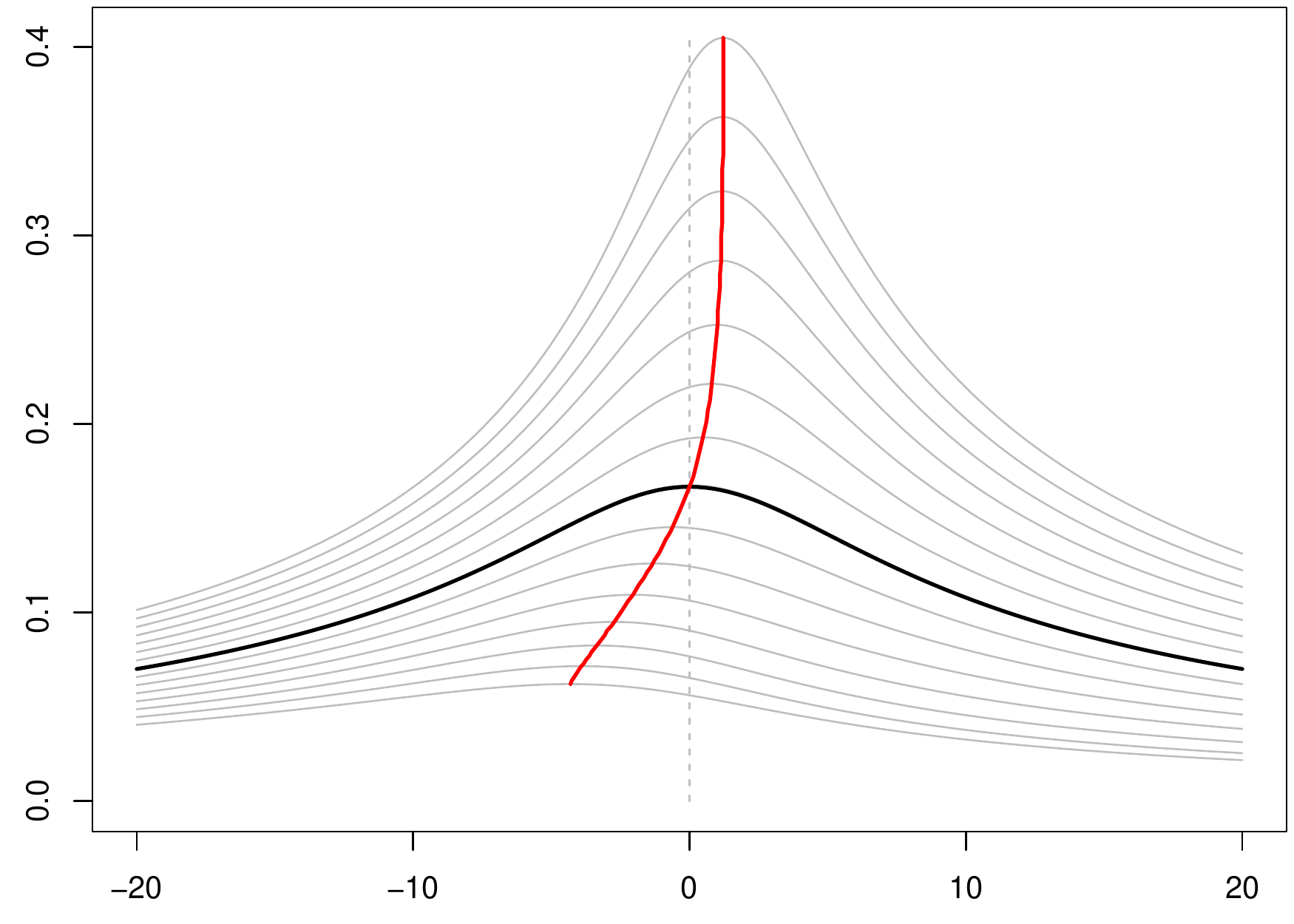}
    \subcaption{Asymmetric $\Tri(1,a)$}
    \label{fig:tri_asym}
    \end{subfigure}
    \caption{\label{fig:tri}
    Illustrations of function $h$ for  $\Tri(a,a)$ (left), and $\Tri(1,a)$ (right) distributions, with $a$ varying   from $0$  to $2$. The  black  curve  represents  the $h$ function  of the  $\Tri(1,1)$ distribution. Maxima of $h$ corresponding to the optimal proxy variances are depicted by the red curve.  \xaxis}
\end{figure}

\subsection{Uniform distribution\label{sec:uniform}}

In this section, we prove that the uniform distribution is strictly sub-Gaussian using a similar proof as for obtaining the optimal proxy variance in the Bernoulli case (i.e., Proposition~\ref{prop:bernoulli}). First, we observe that after translation/dilatation, we may always reduce the problem to the case of $X\sim \Unif([0,1])$. In this case, the moment generating function is straightforward to compute and we may write
\begin{equation}
    h(\lambda)=\frac{2}{\lambda^2}\ln\left[\frac{2 \text{ sh}(\frac{\lambda}{2})}{\lambda}\right],
\end{equation}
which is a symmetric and is a $\mathcal{C}^{\infty}$ function. It remains to prove that it attains its global maximum at $\lambda=0$. 

To this end, we use Proposition~\ref{prop:ODEs} and the remark following it.  ODEs~\eqref{eq:ODE-first-order} and~\eqref{eq:ODE-second-order} are respectively
\begin{equation}
    h'(\lambda)+\frac{2}{\lambda}h(\lambda)=r(\lambda)=\frac{\text{ ch}\left(\frac{\lambda}{2}\right)}{\lambda^2\text{ sh}\left(\frac{\lambda}{2}\right)} -\frac{2}{\lambda^3}, \text{ with } h(0)=\Var[X]=\frac{1}{12},
\end{equation}
and
\begin{equation}\label{UniformEDO}
    h''(\lambda)+\frac{3}{\lambda}h'(\lambda)=r'(\lambda)+\frac{r(\lambda)}{\lambda}=
    -\frac{2 s(\lambda)}{\lambda^4 \text{ sh}^2(\frac{\lambda}{2})},
\end{equation}
with $h(0)=\frac{1}{12}$, $h'(\lambda)=0$, and $s(\lambda)\coloneqq \lambda^2+\lambda \text{ sh}(\lambda)-4\text{ch}(\lambda)+4$. 
We now apply the same method as for the Bernoulli case. That is, we need to show that the r.h.s. of~\eqref{UniformEDO} is negative, which amounts to showing that the function $s$ is positive. 
Let us first observe that the result holds locally around $\lambda=0$. Indeed, we have $h(0)=\frac{1}{12}$, $h'(0)=0$ and $h''(0)=-\frac{1}{720}<0$. 
Then, observe that
\bea s(\lambda)&=&\lambda^2+\lambda\text{ sh}(\lambda)-4\text{ ch}(\lambda)+4,\cr
s'(\lambda)&=&2\lambda-3\text{ sh}(\lambda)+\lambda \text{ ch}(\lambda),\cr
s''(\lambda)&=&2-2\text{ ch}(\lambda)+\lambda \text{ sh}(\lambda),\cr
s^{(3)}(\lambda)&=&\lambda \text{ ch}(\lambda)-\text{ sh}(\lambda)\text{, and}\cr
s^{(4)}(\lambda)&=&\lambda \text{ sh}(\lambda),
\eea
from which we immediately conclude that
\begin{equation} s(0)=0,\, s'(0)=0,\, s''(0)=0,\, s^{(3)}(0)=0 \text{, and } s^{(4)}(0)=0.
\end{equation}
Obviously $s^{(4)}$ is strictly positive on $\mathbb{R}^*$, thus $s^{(3)}$ is a strictly increasing function on $\mathbb{R}$. Since $s^{(3)}(0)=0$, we conclude that $s^{(3)}$ is strictly negative on $\mathbb{R}_-^*$ and strictly positive on $\mathbb{R}_+^*$. Thus, $s''$ is strictly decreasing on $\mathbb{R}_-^*$ and strictly increasing on $\mathbb{R}_+^*$. Finally, since $s''(0)=0$, we conclude that $s''$ is strictly positive on $\mathbb{R}^*$, therefore $s'$ is a strictly increasing function on $\mathbb{R}$. Since $s'(0)=0$ then $s'$ is strictly negative on $\mathbb{R}_-^*$ and strictly positive on $\mathbb{R}_+^*$ so that $s$ is strictly decreasing on $\mathbb{R}_-^*$ and strictly increasing on $\mathbb{R}_+^*$. Since $s(0)=0$, we conclude that $s$ is strictly positive on $\mathbb{R}^*$.
This proves that $h$ has only one unique critical point, which is therefore the global maximizer. 

In conclusion for $X\sim \Unif([a,b])$ with $a<b,$ we have the celebrated result that
\beqq \sigma_\opt^2=\Var[X]=\frac{1}{12}(b-a)^2.\eeqq

\subsubsection{Sum of independent uniform variables}
We may now consider the sum of independent (but not necessarily identically distributed) uniform random variables. Let $(X_1,\dots,X_n)$ be independent variables with $X_i\sim \Unif([a_i,b_i])$ for  $i\in \{ 1,\ldots,n\}$, with $a_i<b_i$ and denote $S_n=\underset{i=1}{\overset{n}{\sum}} X_i$. Since the family of uniform distributions is invariant under translation and multiplication by a constant, we have the standard result that
\beqq Z_i=\frac{X_i-a_i}{b_i-a_i}\sim \Unif([0,1]).\eeqq 
Thus, since $X_i=(b_i-a_i)Z_i+a_i$, we have
\beqq h_{X_i}(\lambda)=(b_i-a_i)^2 \,h_{\Unif([0,1])}((b_i-a_i)\lambda),\eeqq
and then, since the $h$ function of a sum of independent random variables is the sum of the $h$ functions of the variables and by independence of 
the variables $(X_i)_{i\leq n}$, we obtain
\beqq h_{S_n}(\lambda)=\sum_{i=1}^n (b_i-a_i)^2 h_{\Unif([0,1])}((b_i-a_i)\lambda).\eeqq
The  sum of the r.h.s. of the equation above is composed of functions that are all strictly increasing on $\mathbb{R}_-$ and all strictly decreasing 
on $\mathbb{R}_+$. Thus, it too is strictly increasing on $\mathbb{R}_-$ and strictly decreasing on $\mathbb{R}_+$. In particular, the global maximum is unique and obtained at $\lambda=0$ for which we find:
\beqq \sigma_\opt^2[S_n]=\sum_{i=1}^n \Var[X_i]=\frac{1}{12}\sum_{i=1}^n(b_i-a_i)^2\eeqq

\begin{remark}
Note in particular that the sum of two independent uniform variables is generically a trapezoid distribution, with symmetric triangular parts, or a symmetric (up to translation) triangular distribution. However the general asymmetric triangular case, considered in Section~\ref{sec:triangular}, cannot be expressed as a sum of independent uniform distributions.
\end{remark}

\subsection{Kumaraswamy distribution}\label{sec:kuma}

Kumaraswamy distribution is characterized by the density on $(0,1)$:
\begin{align*}
    f(x;\alpha,\beta)=\alpha\beta x^{\alpha-1}(1-x^\alpha)^{\beta-1},
\end{align*}
for $\alpha,\beta>0$, which yields the simple distribution function of form
\begin{align*}
    F(x;\alpha,\beta)=1-(1-x^\alpha)^{\beta}.
\end{align*}
The distribution was first studied in \cite{kumaraswamy1980generalized} and was considered in details by \cite{jones2009kumaraswamy}. 

\begin{proof}[Proof of Proposition~\ref{prop:symmetry-strict} for the Kumaraswamy distribution]$ $\newline
The Kumaraswamy distribution is symmetric if and only if $\alpha=\beta=1$  \citep{jones2009kumaraswamy}. In this case, it reduces to the uniform distribution, which is strictly sub-Gaussian, as was proved in Section~\ref{sec:uniform}.

Conversely, let us now consider any potentially strictly sub-Gaussian Kumaraswamy distribution. It must then satisfy the necessary conditions of Proposition~\ref{prop:skewness}. The third cumulant $\kappa_3$ vanishes if and only if the parameters satisfy the relation
\begin{align}\label{eq:zero-skewness-kuma}
    \alpha=\frac{1}{\beta-(\beta-1)2^{\frac{1}{\beta}}}.
\end{align}
In such a case, a numerical evaluation of the $4^{\text{th}}$ cumulant $\kappa_4=\E[(X-\E[X])^4]- 3\Var[X]^2$ demonstrates that it is negative, thus both necessary conditions of Proposition~\ref{prop:skewness} hold. However, a numerical evaluation also shows that the maximizer of the $h$ function is never located at zero (i.e., the condition of Corollary~\ref{cor:NSC-strict-h-max-0} is not satisfied), except for $\alpha=\beta=1$ (i.e., the uniform distribution). This is illustrated in Figure~\ref{fig:kuma}, where the function $h$ is plotted for $(\alpha,\beta)$, satisfying relation~\eqref{eq:zero-skewness-kuma}, with $\beta$ varying in the interval $[10^{-3}, 5]$. The maximum of $h$ is illustrated with the red curve, showing that the global maximizer always deviates from zero, except for the case of the uniform distribution (the black curve) and the degenerate symmetric Bernoulli distribution (the blue curve). 
This proves the necessity of symmetry and concludes the proof.
\end{proof}

\begin{figure}[!ht]
    \centering
    \includegraphics[width=.5\textwidth]{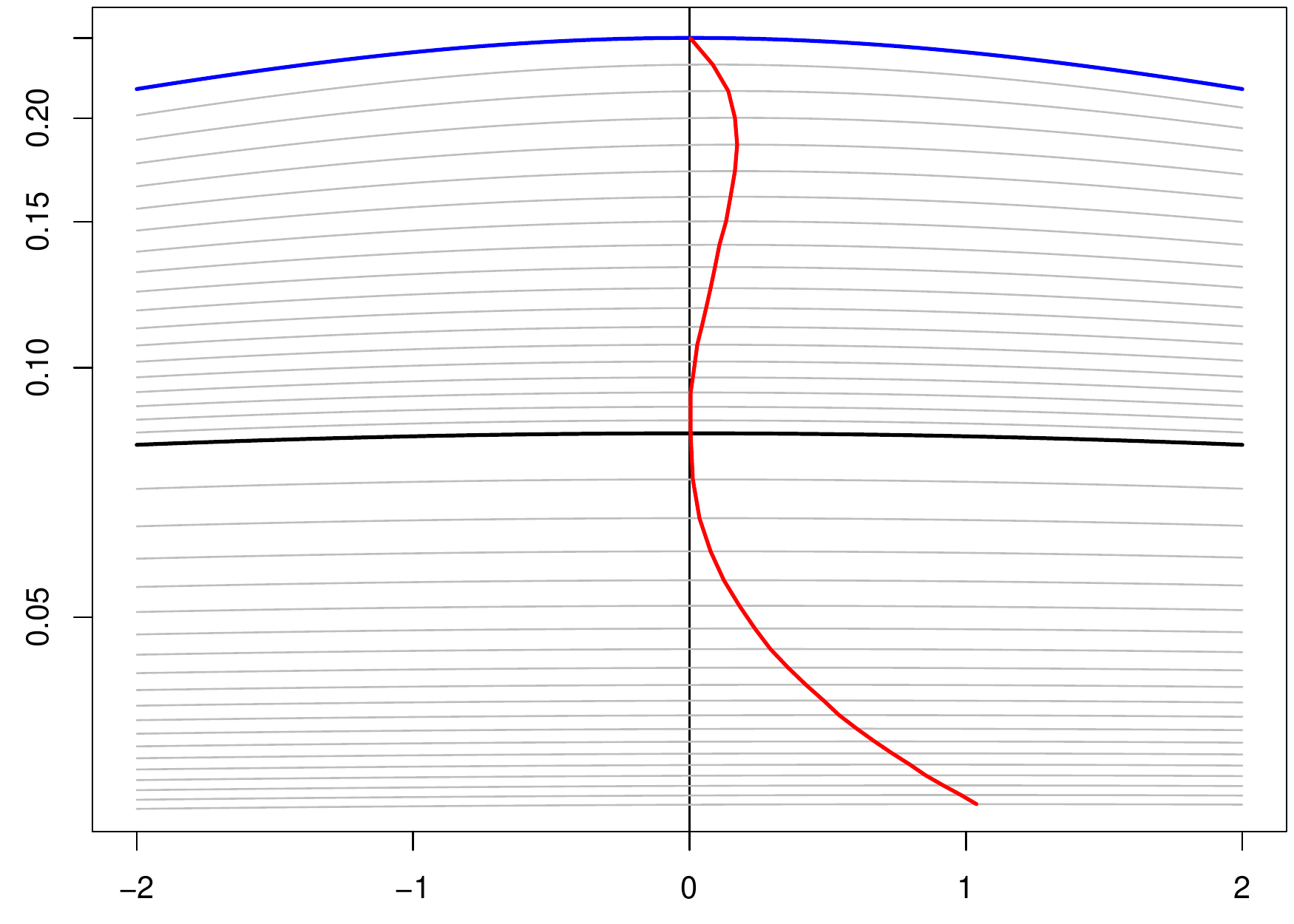}
    \caption{\label{fig:kuma}
    Function $h$ for Kumaraswamy distributions for $(\alpha,\beta)$ satisfying relation~\eqref{eq:zero-skewness-kuma}, with $\beta$ varying from $10^{-3}$ (top gray curve) to $5$ (bottom gray curve). The black curve represents the $h$ function of the uniform distribution. \bernoullitext 
    Maxima of $h$, corresponding to the optimal proxy variances, are depicted by the red curve. 
    In particular, the maxima are located at zero only for the symmetric distributions (i.e., the uniform distribution and Ber$\left(\frac{1}{2}\right)$). \xaxis Log-scale on the $y$-axis.  }
\end{figure}
\pagebreak

\subsection{Beta distribution\label{sec:beta}}

The optimal proxy variance for the $\Beta(\alpha,\beta)$  distribution was derived in \cite{marchal2017sub}, Theorem 2.1. In particular, this theorem states that the optimal proxy variance is equal to the variance if and only $\alpha=\beta$. That is, if and only if the beta distribution is symmetric. This proves Proposition~\ref{prop:symmetry-strict} for the beta distribution. The $h$ function and the optimal proxy variance is illustrated on Figure~\ref{fig:beta}. 

\begin{figure}[!ht]
    \centering
    \begin{subfigure}{0.49\textwidth}
            \centering
            \includegraphics[width=\textwidth]{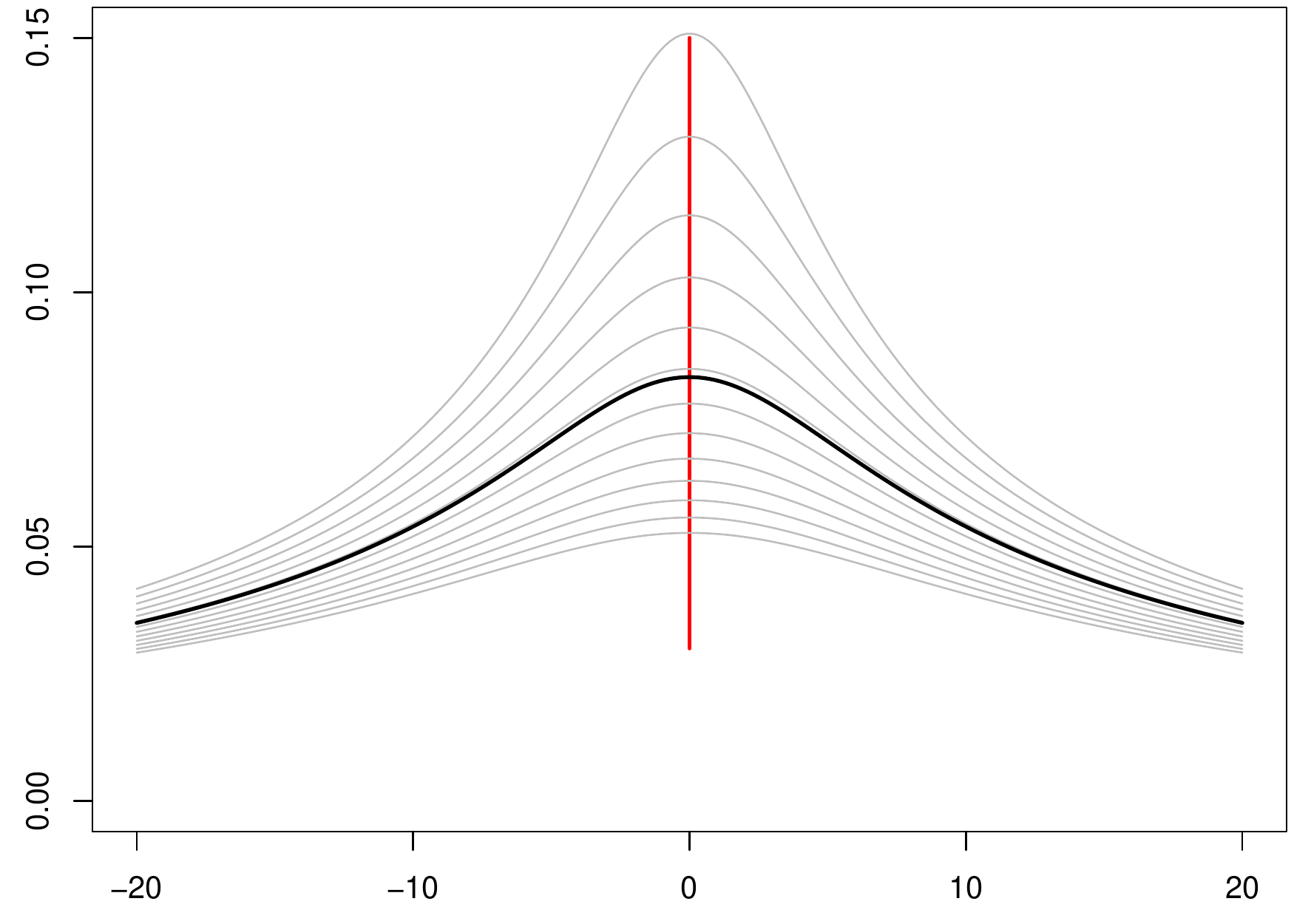}
    \subcaption{Symmetric $\Beta(a,a)$}
    \label{fig:beta_sym}
    \end{subfigure}
    \begin{subfigure}{0.49\textwidth}
            \centering
            \includegraphics[width=\textwidth]{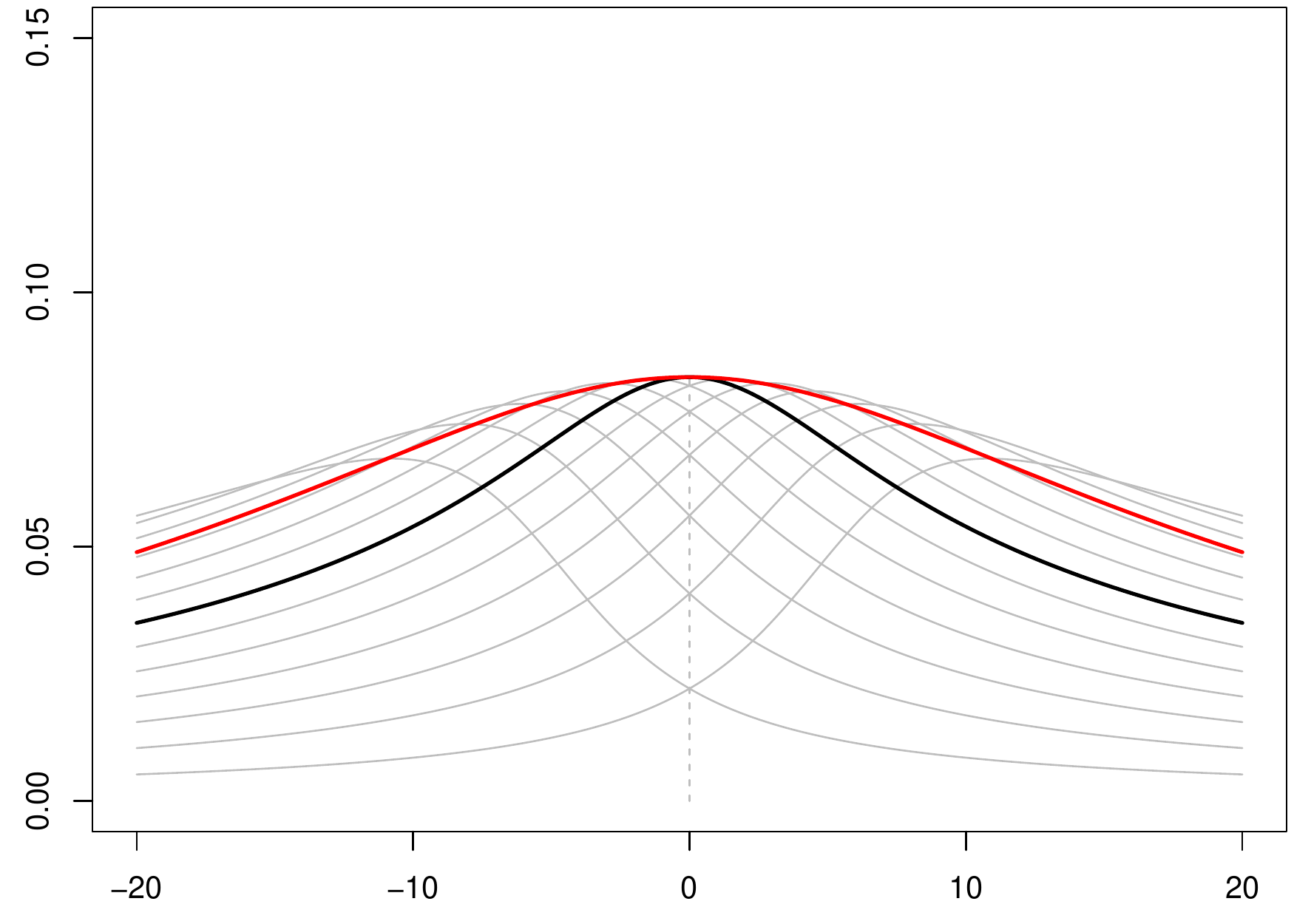}
    \subcaption{Asymmetric $\Beta(a,2-a)$}
    \label{fig:beta_asym}
    \end{subfigure}
    \caption{\label{fig:beta}
    Illustrations of function $h$ for  $\Beta(a,a)$ (left), and $\Beta(a,2-a)$ (right) distributions, with $a$ varying   from $0.2$  to $1.8$. The black curve represents the $h$ function of the uniform distribution. Maxima of $h$ corresponding to the optimal proxy variance are depicted by the red curve.  \xaxis}
\end{figure}
\pagebreak

\appendix

\section{Technical results\label{sec:appendix}}

\subsection[A lemma regarding the supremum of h]{A lemma regarding the supremum of $h$}
\begin{lemma}\label{lem:sup-equals-max}
	For variables with bounded support, the supremum in \eqref{eq:supremum} is a maximum.
\end{lemma}
\begin{proof}
	We have
$$ \exp(\lambda[X-\mu])\leq \exp(|\lambda|[M+|\mu|])\leq \exp(2|\lambda|M),$$
where $M$ is the maximum value of $|X|$ (which is finite since $X$ is almost surely bounded). Thus we obtain
\begin{equation}
    \frac{2}{\lambda^2}\mathcal{K}(\lambda)\leq \frac{4M|\lambda|}{\lambda^2}=\frac{4M}{|\lambda|}\underset{\lambda\to \infty}{\to} 0.
\end{equation}
Therefore, the supremum is not at infinity and since the function $\lambda \in \mathbb{R}\mapsto \frac{2\mathcal{K}(\lambda)}{\lambda^2}$ is continuous and positive, it must achieve its maximal value at finite values of $\lambda$. 
\end{proof}

\subsection{Proof of Proposition~\ref{prop:sigma_opt-NSC}\label{sec:proof-Delta}}

The proof of Proposition~\ref{prop:sigma_opt-NSC} is based on the study of the variations of the $\Delta$ function, defined in Equation~\eqref{eq:Delta}, which is the object of the next lemma.

We first observe that for any $\lambda\in \mathbb{R}^*$, the function $\sigma^2\mapsto \Delta(\sigma^2,\lambda)$ is strictly increasing. Moreover, at $\lambda=0$ we have:
\begin{equation}
    \Delta(\sigma^2,\lambda)=\left(\sigma^2-\Var[X]\right)\frac{\lambda^2}{2}-\E[(X-\mu]^3)\frac{\lambda^3}{3!}+\left(3\sigma^4-\E[(X-\mu)^4]\right)\frac{\lambda^4}{4!}+O(\lambda^5).
\end{equation}
Therefore, for $\sigma^2>\Var[X]$, the function $\lambda\mapsto \Delta(\sigma^2,\lambda)$ is strictly positive in a neighborhood of $\lambda=0$, while for $\sigma^2<\Var[X]$, the function is strictly negative in a neighborhood of $\lambda=0$. 
Thus we obtain the following lemma.
\begin{lemma}\label{lem:variations-of-Delta}
The variations of $\lambda\mapsto \Delta(\sigma^2,\lambda)$ depend on the value of $\sigma^2$, with respect to $\sigma_\opt^2$, as follows:
\begin{enumerate}
    \item for $\sigma^2>\sigma^2_\opt$, the function $\lambda\mapsto \Delta(\sigma^2,\lambda)$ is strictly positive on $\mathbb{R}$,
    \item for $\sigma^2=\sigma^2_\opt$, the function $\lambda\mapsto \Delta(\sigma^2_\opt,\lambda)$ is non-negative and there exists at least one point $\lambda_0\in \mathbb{R}$ for which $\Delta(\sigma^2_\opt,\lambda_0)=0$. In particular, since the function remains non negative, it implies that $\partial_\lambda\Delta(\sigma^2_\opt,\lambda_0)=0$ and $\partial^2_\lambda\Delta(\sigma^2_\opt,\lambda_0)\geq 0$, and
    \item for $\sigma^2<\sigma^2_\opt$, there exists an interval  not reduced to a point on which  $\lambda\mapsto \Delta(\sigma^2,\lambda)<0$,
\end{enumerate}
where the first and second derivatives of $\lambda\mapsto \Delta(\sigma^2,\lambda)$ are denoted by $\partial_\lambda\Delta$ and $\partial^2_\lambda\Delta$, respectively.
\end{lemma}

\begin{proof}
The proof is based on fact that $\sigma^2 \mapsto \Delta(\sigma^2,\lambda)$ is strictly increasing, and the fact that  $(\sigma^2,\lambda) \mapsto \Delta(\sigma^2,\lambda)$ is continuous. 

1. Assume by contradiction that $\Delta$ is not strictly positive. Then, since it is non-negative, there must exist at least one point $\lambda_1$ for which $\Delta(\sigma^2,\lambda_1)=0$ with $\lambda_1\neq 0$ (because for $\sigma^2>\sigma^2_\opt\geq \Var[X]$, we know that $\Delta$ is strictly positive around $\lambda=0$). Thus, for any $\tilde{\sigma}^2<\sigma^2$ we have $\Delta(\tilde{\sigma}^2,\lambda_1)<\Delta(\sigma^2,\lambda_1)=0$, so that $X$ is not $\tilde{\sigma}^2$-sub-Gaussian, hence $\tilde{\sigma}^2< \sigma^2_\opt$ and by taking the limit $\tilde{\sigma}^2\to \sigma^2$ from below, we get $\sigma^2\leq \sigma^2_\opt$, which is a contradiction.

2. At $\sigma^2=\sigma^2_\opt$, the function $\Delta$ must vanish at least at one point $\lambda_0\in \mathbb{R}$, while remaining non-negative on $\mathbb{R}$. Indeed, the function $\Delta$ is non-negative by the sub-Gaussian definition, but if $\Delta$ was strictly positive on $\mathbb{R}$, then the continuity of $\Delta$, relatively to ($\sigma^2,\lambda)$, would imply that we may lower $\sigma^2$ without $\Delta$ vanishing. This would be in contradiction with the minimality of $\sigma^2_\opt$.

3. For $\sigma^2<\sigma^2_\opt$, there exists at least a point $\lambda_1\in \mathbb{R}$ for which $\Delta(\sigma^2,\lambda_1)<0$. By continuity of the function $\lambda\mapsto \Delta(\sigma^2,\lambda)$, this implies that there exists a neighborhood of $\lambda_1$ in which $\Delta$ is strictly negative, which concludes the proof. 
\end{proof}

We are now ready to prove Proposition~\ref{prop:sigma_opt-NSC}.

\begin{proof}[Proof of Proposition~\ref{prop:sigma_opt-NSC}]
	Proposition~\ref{lem:variations-of-Delta} indicates that if $\sigma^2=\sigma^2_\opt$, then $\Delta\geq 0$, and there exists $\lambda_0\in \mathbb{R}$, such that $\Delta(\sigma^2,\lambda_0)=0$ and $\partial_\lambda\Delta(\sigma^2,\lambda_0)=0$. Conversely, let us assume that $\Delta\geq 0$ and  $\exists \lambda_0\in \mathbb{R}\text{, such that } \Delta(\sigma^2,\lambda_0) \text{ and } \partial_\lambda\Delta(\sigma^2,\lambda_0)=0 $. Then, since $\Delta\geq 0$, we have $\sigma^2\geq \sigma^2_\opt$, and since for $\tilde{\sigma}^2<\sigma^2$, $\Delta(\tilde{\sigma}^2,\lambda_0)<0$, we also have $\sigma^2\leq \sigma^2_\opt$. Thus $\sigma^2 = \sigma^2_\opt$, which concludes the proof.
\end{proof}

\section*{Acknowledgements}
O.M. would like to thank Universit\'e Lyon $1$, Universit\'e Jean Monnet and Institut Camille Jordan for material support. H.D.N. is  funded by Australian Research Council grants: DE170101134 and DP180101192. This work was supported by the LABEX MILYON (ANR-10-LABX-0070) of Universit\'e de Lyon, within the program "Investissements d'Avenir" (ANR-11-IDEX-0007) operated by the French National Research Agency (ANR). 

\bibliographystyle{apalike}
\bibliography{biblio}

\end{document}